\documentclass[11pt]{amsart}
\usepackage{amssymb}
\usepackage{txfonts}

\makeatletter
\newcommand{\sumprime}{\if@display\sideset{}{'}\sum%
            \else\sum'\fi}
\makeatother

\begin{document}

\numberwithin{equation}{section}

\newtheorem{theorem}{Theorem}[section]
\newtheorem{proposition}[theorem]{Proposition}
\newtheorem{conjecture}[theorem]{Conjecture}
\def\theconjecture{\unskip}
\newtheorem{corollary}[theorem]{Corollary}
\newtheorem{lemma}[theorem]{Lemma}
\newtheorem{observation}[theorem]{Observation}
\newtheorem{definition}{Definition}
\numberwithin{definition}{section} 
\newtheorem{remark}{Remark}
\def\theremark{\unskip}
\newtheorem{question}{Question}
\def\thequestion{\unskip}
\newtheorem{example}{Example}
\def\theexample{\unskip}
\newtheorem{problem}{Problem}

\def\vvv{\ensuremath{\mid\!\mid\!\mid}}
\def\intprod{\mathbin{\lr54}}
\def\reals{{\mathbb R}}
\def\integers{{\mathbb Z}}
\def\N{{\mathbb N}}
\def\complex{{\mathbb C}\/}
\def\dist{\operatorname{dist}\,}
\def\spec{\operatorname{spec}\,}
\def\interior{\operatorname{int}\,}
\def\trace{\operatorname{tr}\,}
\def\cl{\operatorname{cl}\,}
\def\essspec{\operatorname{esspec}\,}
\def\range{\operatorname{\mathcal R}\,}
\def\kernel{\operatorname{\mathcal N}\,}
\def\dom{\operatorname{Dom}\,}
\def\linearspan{\operatorname{span}\,}
\def\lip{\operatorname{Lip}\,}
\def\sgn{\operatorname{sgn}\,}
\def\Z{ {\mathbb Z} }
\def\e{\varepsilon}
\def\p{\partial}
\def\rp{{ ^{-1} }}
\def\Re{\operatorname{Re\,} }
\def\Im{\operatorname{Im\,} }
\def\dbarb{\bar\partial_b}
\def\eps{\varepsilon}
\def\O{\Omega}
\def\Lip{\operatorname{Lip\,}}

\def\Hs{{\mathcal H}}
\def\E{{\mathcal E}}
\def\scriptu{{\mathcal U}}
\def\scriptr{{\mathcal R}}
\def\scripta{{\mathcal A}}
\def\scriptc{{\mathcal C}}
\def\scriptd{{\mathcal D}}
\def\scripti{{\mathcal I}}
\def\scriptk{{\mathcal K}}
\def\scripth{{\mathcal H}}
\def\scriptm{{\mathcal M}}
\def\scriptn{{\mathcal N}}
\def\scripte{{\mathcal E}}
\def\scriptt{{\mathcal T}}
\def\scriptr{{\mathcal R}}
\def\scripts{{\mathcal S}}
\def\scriptb{{\mathcal B}}
\def\scriptf{{\mathcal F}}
\def\scriptg{{\mathcal G}}
\def\scriptl{{\mathcal L}}
\def\scripto{{\mathfrak o}}
\def\scriptv{{\mathcal V}}
\def\frakg{{\mathfrak g}}
\def\frakG{{\mathfrak G}}
\def\texthtbardotlessj
\def\ov{\overline}

\thanks{Supported by NSF grant 11771089 and Gaofeng grant support from School of Mathematical Sciences, Fudan University}

\address{School of Mathematical Sciences, Fudan University, Shanghai, 200433, China}
 \email{boychen@fudan.edu.cn}
 
\title{Weighted Bergman kernel, directional Lelong number and John-Nirenberg exponent}
\author{Bo-Yong Chen}
\date{}
\maketitle
\begin{abstract}
 Let $\psi$ be a plurisubharmonic function on the closed unit ball and $K_{t\psi}(z)$ the Bergman kernel on the unit ball with respect to the weight $t\psi$.  We show that the boundary behavior of $K_{t\psi}(z)$ is determined by certain directional Lelong number of $\psi$ for all $t$ smaller than the John-Nirenberg exponent of $\psi$ associated to certain family of nonisotropic balls, which is always positive.
\end{abstract}

\bigskip

\section{Introduction}

Let $B_1$ be the unit ball in ${\mathbb C}^n$ and $PSH(B_1)$ the set of plurisubharmonic (psh) functions on $B_1$ (we always assume that psh functions are not identically $-\infty$). For each  $\psi\in PSH(B_1)$ we define $K_{t\psi}(z,w)$ to be the weighted Bergman kernel of the Hilbert space
  $$
  A^2_{t\psi}=\left\{f\in {\mathcal O}(B_1): \int_{B_1} |f|^2 e^{-t\psi}<\infty\right\},\ \ \ t\ge 0.
  $$
  Set $K_{t\psi}(z)=K_{t\psi}(z,z)$. A cerebrated theorem of Demailly \cite{Demailly92} states that
  $$
  \psi_t:=\frac1t \log K_{t\psi}(z)\rightarrow \psi(z)\ \ \ (t\rightarrow +\infty)
  $$
  and
  $$
  \nu(\psi,z)-2n/t\le \nu(\psi_t,z)\le \nu(\psi,z)
  $$
  where $\nu(\varphi,z)$ denotes the Lelong number for a psh function $\varphi$ at $z$.

  In this paper we consider the case when $t$ is fixed and $z$ approaches the boundary $\partial B_1$. We discover that for all sufficiently small $t$ the asymptotic behavior of $K_{t\psi}(z)$ at a boundary point $\zeta$ is determined by certain\/ {\it directional\/} Lelong number of $\psi$ at $\zeta$. To state the results precisely, we need to introduce some notions. Let
   $$
   \tilde{B}_r=\left\{(z_1,z')\in {\mathbb C}\times {\mathbb C}^{n-1}: |z_1|<r,|z'|<\sqrt{r}\right\}.
   $$
     For a bounded domain $\Omega\subset {\mathbb C}^n$ we define  $\tilde{\mathcal B}(\Omega)$ to be the set of all $F(\tilde{B}_r)\subset \Omega$ where $F$ is a complex affine mapping composed by a translation and a unitary transformation. We define the John-Nirenberg exponent of $\psi$ associated to the family  $\tilde{\mathcal B}(\Omega)$ by
       $$
  \tilde{ \varepsilon}_\Omega(\psi):=\sup\left\{\varepsilon:\sup_{D\in \tilde{\mathcal B}(\Omega)} \fint_D e^{\varepsilon |\psi-\psi_D|}<\infty\right\}
  $$
  where $\psi_D=\fint_D \psi$ is the mean value of $\psi$ over $D$.
   For each $\zeta\in \partial B_1$ we denote by ${\mathcal T}_\zeta$ the holomorphic tangent space at $\zeta$ and ${\mathcal N}_\zeta$ the orthogonal complement of ${\mathcal T}_\zeta$ in ${\mathbb C}^n$. Let $F_\zeta$ be the complex affine mapping which is composed by a translation and a unitary transformation, and maps the $z_1$ axis to ${\mathcal N}_\zeta$ and $z_1=0$ to ${\mathcal T}_\zeta$ respectively.

     \begin{theorem}\label{th:Bergman_Lelong}
  Let $\psi$ be a psh function in a neighborhood of the closed ball $\overline{B}_R:=\{|z|\le R\}$ where $R>1$. For each $0\le t<\tilde{\varepsilon}_{B_R}(\psi)$ and each $\zeta\in \partial B_1$ we have
  \begin{equation}\label{eq:BergmanVsLelong}
   \lim_{r\rightarrow 0} \frac{\log K_{t\psi}((1-r)\zeta)}{\log 1/r}=n+1-t\tilde{\nu}(\psi_\zeta)
  \end{equation}
  where $\psi_\zeta=\psi\circ F_\zeta$ and
  $$
 \tilde{\nu}(\psi_\zeta)=\lim_{r\rightarrow 0} \frac{1}{\log r}\sup_{\theta_1,\cdots,\theta_n}\psi_\zeta(re^{i\theta_1},\sqrt{r} e^{i\theta_2},\cdots,\sqrt{r} e^{i\theta_n}).
  $$
     \end{theorem}

Note that the quantity $2^{1-n}\tilde{\nu}(\psi_\zeta)$ is essentially the directional Lelong number with coefficients $(1,2,\cdots,2)$ of $\psi_\zeta$ at\/ $0$ $($see \cite{DemaillyBook},\,p.\,166$)$.

Of course, Theorem \ref{th:Bergman_Lelong} is meaningless unless one has verified the following

\begin{theorem}\label{th:JohnNirenberg}
 Let $\Omega$ be a bounded domain in ${\mathbb C}^n$. If $\psi$ is psh in a neighborhood of $\overline{\Omega}$ then $ \tilde{ \varepsilon}_\Omega(\psi)>0$.
\end{theorem}

The proof of Theorem \ref{th:JohnNirenberg} relies on the following local Bernstein type inequality
 \begin{equation}\label{eq:Bernstein_0}
   \sup_B \psi\le \sup_E \psi+ C_{n,\alpha} \left( 1+|\psi_{B_R}|\right)^\alpha \left[1+\log \left({|B|}/{|E|}\right)\right]
   \end{equation}
 for each ball $B\subset B_R$, measurable set $E\subset B$ and negative psh function $\psi $ on $B_{2R}$, where $\alpha>1$ and $C_{n,\alpha}$ depends only on $n,\alpha$. Inequalities like (\ref{eq:Bernstein_0}) were obtained earlier by Brudnyi \cite{Brudnyi}. Analogous global inequalities were obtained by Benelkourchi et al. \cite{Zeriahi} for the Lelong class of psh functions.  The analysis in these papers relies heavily on (nonlinear) pluripotential theory. Here we shall present an entirely new approach, using only linear analysis: the Riesz decomposition theorem and some basic facts from the theory of weights   (see \cite{SteinHarmonicBook}).

       \section{Proof of Theorem \ref{th:Bergman_Lelong}}

    Theorem \ref{th:Bergman_Lelong} will be deduced from the following

        \begin{theorem}\label{th:BergmanEstimate}
 Let $\psi$ be a psh function in a neighborhood of $\overline{B}_R$ where $R>1$. For each $0\le t<\tilde{\varepsilon}_{B_R}(\psi)$ there exists a constant $C>0$ such that for all $\zeta\in \partial B_1$ and $0<r\ll 1$,
  \begin{equation}\label{eq:Bergman}
  C^{-1} r^{-n-1} e^{t\psi_{\tilde{B}_r(\zeta)}}\le K_{t\psi}((1-r)\zeta) \le C r^{-n-1} e^{t\psi_{\tilde{B}_r(\zeta)}}
   \end{equation}
   where $\tilde{B}_r(\zeta)=F_\zeta(\tilde{B}_r)$ and $F_\zeta$ is as in \S\,1.
     \end{theorem}

    We start with a few elementary lemmas. For each $\zeta\in B_1$ we denote by $T_\zeta$ the holomorphic automorphism of $B_1$ which maps $\zeta$ onto the origin.

    \begin{lemma}\label{lm:Ball}
    Let $0<r<1$ and $\zeta_r:=(1-r,0,\cdots,0)$. Then we have
    \begin{equation}\label{eq:Ball_1}
     D_r:=\left\{z:|z_1-(1-r)|<r/2,|z'|<\sqrt{r/2}\right\}\subset B_1
    \end{equation}
    \begin{equation}\label{eq:Ball_2}
    \left\{z:|T_{\zeta_r}(z)|<1/\sqrt{2}\right\}\subset D_r':=\left\{z:|z_1-1|<10r,|z'|<\sqrt{20r}\right\}.
        \end{equation}
    \end{lemma}

    \begin{proof}
    If $z\in D_r$, then $|z_1|<1-r/2$, so that
    $$
    |z|^2<(1-r/2)^2+r/2=1-r/2+r^2/4<1,
    $$
    i.e. (\ref{eq:Ball_1}) is verified.

    Now suppose $|T_{\zeta_r}(z)|^2<1/2$. By the standard formula
    \begin{equation}\label{eq:Rudin}
    1-|T_{\zeta_r}(z)|^2=\frac{(1-|z|^2)(1-|\zeta_r|^2)}{|1-\langle z,\zeta_r\rangle|^2}
    \end{equation}
    (see \cite{Zhu05}, p.\,5), we conclude that
    \begin{eqnarray*}
     2|1-z_1| & \ge & 2(1-|z_1|)\ge 1-|z_1|^2\ge 1-|z|^2\\
     & \ge & \frac12\cdot\frac{|1-(1-r)z_1|^2}{2r-r^2}\\
     & \ge & \frac{(|1-z_1|-r|z_1|)^2}{4r-2r^2}.
    \end{eqnarray*}
    With $c_1:=r|z_1|$ and $c_2:=4r-2r^2$ we have
    $$
    |1-z_1|\le c_1+c_2+\sqrt{2c_1 c_2+c_2^2}<10 r.
    $$
    On the other hand, we have
    $$
    |z'|^2<1-|z_1|^2\le 2|1-z_1|<20 r.
    $$
    Thus (\ref{eq:Ball_2}) is verified.
    \end{proof}

      \begin{lemma}\label{lm:BMO_PSH_4}
 Let $V$ be a measurable set in ${\mathbb R}^n$ and $\psi\in L^1(V)$.  For each measurable set $W\subset V$  we have
  \begin{equation}\label{eq:BMO_1}
 |\psi_W-\psi_V | \le  \frac{|V|}{|W|}\fint_V |\psi-\psi_V|
 \end{equation}
 \begin{equation}\label{eq:BMO_2}
 \fint_W |\psi-\psi_W|\le  \frac{2|V|}{|W|}\fint_V |\psi-\psi_V|.
 \end{equation}
   \end{lemma}

  \begin{proof}
  First of all, we have
  \begin{eqnarray*}
  |\psi_W-\psi_V | & \le &  \frac1{|W|}\int_{W} |\psi-\psi_{V}|\\
  & \le &  \frac1{|W|}\int_{V} |\psi-\psi_V|\\
  & \le &  \frac{|V|}{|W|}\fint_V |\psi-\psi_V|.
  \end{eqnarray*}
  Next we have
  \begin{eqnarray*}
   \fint_{W} |\psi-\psi_W| & \le &  \fint_{W} \left[|\psi-\psi_V|+|\psi_W-\psi_V|\right]\\
   & \le & \frac{2|V|}{|W|}\fint_V |\psi-\psi_V|.
  \end{eqnarray*}
  \end{proof}

  \begin{lemma}\label{lm:ReverseJensen}
   Let $\psi$ be a psh function in a neighborhood of $\overline{B}_R$ where $R>1$. For each $0\le t<\tilde{\varepsilon}_{B_R}(\psi)$ there exists a constant $C>0$ such that
   \begin{equation}\label{eq:ReverseJensen}
  \fint_D e^{{t \psi}}\le C e^{t \psi_{D}} \ \ \ {and\ \ \ } \fint_D e^{-{t \psi}}\le C e^{-t \psi_{D}}
 \end{equation}
 for all $D\in \tilde{\mathcal B}(B_R)$.
  \end{lemma}

  \begin{proof}
   From the inequality
   $$
   \fint_D e^{{t |\psi-\psi_D|}}\le C,
   $$
   we immediately get (\ref{eq:ReverseJensen}).
  \end{proof}

  \begin{remark}
   By  Jensen's inequality
  \begin{equation}
e^{t \psi_{D}}\le  \fint_D e^{{t \psi}} \ \ \ { and\ \ \ }  e^{-t \psi_{D}}\le  \fint_D e^{-{t \psi}},
  \end{equation}
  one may call (\ref{eq:ReverseJensen}) the reverse Jensen inequality.
  \end{remark}

    \begin{proof}[Proof of Theorem \ref{th:BergmanEstimate}]
     Without loss of generality, we assume that $\zeta=(1,0,\cdots,0)$.  Set $\zeta_r=(1-r)\zeta$.    For each $f\in A^2_{t\psi}$ we have
  \begin{eqnarray*}
   |f(\zeta_r)|   \le   \fint_{D_r}|f| & \le & \left[\fint_{D_r}|f|^2 e^{-t\psi}\right]^{1/2}\left[\fint_{D_r} e^{t\psi}\right]^{1/2} \\
  & \le & \left[\int_{B_1} |f|^2 e^{-t\psi}\right]^{1/2}\left[ \frac1{|D_r|}\fint_{D_r} e^{t\psi}\right]^{1/2} \\
  &\le & \left[\int_{B_1} |f|^2 e^{-t\psi}\right]^{1/2}\left[ C r^{-n-1} e^{t\psi_{D_r}} \right]^{1/2}
       \end{eqnarray*}
 where the last inequality follows from Lemma \ref{lm:ReverseJensen}.
  As
 $$
 \left|\psi_{D_r}-\psi_{\tilde{B}_r(\zeta)}\right|\le C, \ \ \ r\ll1
 $$
 in view of Lemma \ref{lm:BMO_PSH_4}, we have
 $$
 K_{t\psi}(\zeta_r)\le \sup_{f\in A^2_{t\psi}}\frac{|f(\zeta_r)|^2}{\|f\|_{r\psi}^2}\le C r^{-n-1} e^{t\psi_{\tilde{B}_r(\zeta)}}.
 $$

 For the lower bound of $K_{t\psi}$ we shall use $L^2-$estimates of the $\bar{\partial}-$equation in a standard way (compare \cite{Chen99}). Let $g(z,w)$ be the pluricomplex Green function of $B_1$ with pole at $w$, i.e.
 $$
 g(z,w)=\log |T_w(z)|
 $$
 where $T_w$ is the holomorphic automorphism of $B_1$ which maps $w$ onto the origin. Set $g_r(z)=g(z,\zeta_r)$.
  Choose a smooth cut-off function $\chi:{\mathbb R}\rightarrow [0,1]$ such that $\chi|_{(-\infty,-1]}=1$ and $\chi|_{[0,\infty)}=0$.  By the Donnelly-Fefferman  estimate (see e.g. \cite{BerndtssonCharpentier}), we may find a solution of
  $$
  \bar{\partial} u=\bar{\partial}\left[\chi\left(-\log(-g_r)+\log\log \sqrt{2}\right)\right]=:v
  $$
  which satisfies
  \begin{eqnarray*}
   \int_{B_1} |u|^2 e^{-t\psi-2ng_r} & \le & C_0  \int_{B_1} |v|^2_{-i\partial\bar{\partial}\log(-g_r)} e^{-t\psi-2ng_r} \\
   &\le & C_{n} \int_{\{|T_{\zeta_r}|< 1/\sqrt{2}\}}  e^{-t\psi}\\
   & \le & C_n \int_{D_r'} e^{-t\psi}
      \end{eqnarray*}
      in view of Lemma \ref{lm:Ball},
   where $C_0$ is a universal constant and $C_{n}$ depends only on $n$.
   Set
   $$
   f:=\chi\left(-\log(-g_r)+\log\log \sqrt{2}\right)-u.
   $$
    It follows that
   $
   f\in {\mathcal O}(B_1)$, $f(\zeta_r)=1$, and
   $$
   \int_{B_1} |f|^2 e^{-t\psi}\le C_n \int_{D_r'} e^{-t\psi}.
   $$
   Thus
   \begin{equation}\label{eq:Bergman_1}
   K_{t\psi}(\zeta_r)\ge \frac{|f(\zeta_r)|^2}{\|f\|^2_{t\psi}}\ge C_n^{-1}\left[\int_{D_r'}  e^{-{t\psi}}\right]^{-1}\ge C^{-1} r^{-n-1} e^{t\psi_{D_r'}}
   \end{equation}
   in view of Lemma \ref{lm:ReverseJensen}. As
 $$
 \left|\psi_{D_r'}-\psi_{\tilde{B}_r(\zeta)}\right|\le C
 $$
 in view of Lemma \ref{lm:BMO_PSH_4},
 we establish the desired lower bound.
  \end{proof}

  \begin{proof}[Proof of Theorem \ref{th:Bergman_Lelong}]
    Set $\tilde{P}_r=\{z:|z_1|<r,\max_{j\ge 2}\,|z_j|<\sqrt{r}\}$. Then we have
   \begin{eqnarray*}
    \liminf_{r\rightarrow 0} {\psi_{\tilde{B}_r(\zeta)}}/{\log r} & = &  \liminf_{r\rightarrow 0} {(\psi_\zeta)_{\tilde{B}_r}}/{\log r}\\
    & = &  \liminf_{r\rightarrow 0} (\psi_\zeta)_{\tilde{P}_r}/{\log r}\\
     & \ge & \liminf_{r\rightarrow 0}\frac{1}{\log r}\sup_{\theta_1,\cdots,\theta_n}\psi_\zeta(re^{i\theta_1},r^{1/2}e^{i\theta_2},\cdots,r^{1/2}e^{i\theta_n})\\
    & = & \tilde{\nu}(\psi_\zeta)
   \end{eqnarray*}
   where the second equality follows from  Lemma \ref{lm:BMO_PSH_4} and the inequality follows from the maximum principle for  psh functions.
   On the other hand, for each $r$ we choose 
   $$
   a_r=(re^{i\theta_1(r)},r^{1/2}e^{i\theta_2(r)},\cdots,r^{1/2}e^{i\theta_n(r)})
   $$
    such that
   $$
   \psi_\zeta(a_r)=\sup_{\theta_1,\cdots,\theta_n}\psi_\zeta(re^{i\theta_1},r^{1/2}e^{i\theta_2},\cdots,r^{1/2}e^{i\theta_n}).
   $$
   Set 
   $\tilde{P}_r+a=\{z+a:z\in \tilde{P}_r\}$.
  By Lemma \ref{lm:BMO_PSH_4} we have
      \begin{eqnarray*}
    \limsup_{r\rightarrow 0}{\psi_{\tilde{B}_r(\zeta)}}/{\log r} & = &  \limsup_{r\rightarrow 0} {(\psi_\zeta)_{\tilde{B}_r}}/{\log r}\\
     & = & \limsup_{r\rightarrow 0} {(\psi_\zeta)_{\tilde{P}_r+a_r}}/{\log r}\\
    & \le & \limsup_{r\rightarrow 0} {\psi_\zeta(a_r)}/{\log r}\\
    & = & \tilde{\nu}(\psi_\zeta)
   \end{eqnarray*}
   where the inequality follows from the mean value inequality. Thus by (\ref{eq:Bergman}) we establish (\ref{eq:BergmanVsLelong}).
   \end{proof}

 \section{A local BMO estimate of psh functions}

   A function $\psi\in L^1_{\rm loc}(\Omega)$ is of BMO (bounded mean oscillation) if
 \begin{equation}\label{eq:BMO}
 \|\psi\|_{{\rm BMO}(\Omega)}:=\sup_B \fint_B |\psi-\psi_B|<\infty,
 \end{equation}
 where the supremum is taken over all balls $B\subset\subset \Omega$.
 BMO was first introduced by John-Nirenberg \cite{JohnNirenberg} in connection with PDE, who also proved a crucial inequality:
 \begin{equation}\label{eq:JohnNirenberger}
  \sup_{B\subset \Omega}\fint_B e^{{c_n|\psi-\psi_B|}/\|\psi\|_{{\rm BMO}(\Omega)}}\le C_n
  \end{equation}
  where  $c_n,C_n>0$ depend only on $n$.
  The BMO space  became well-known after Fefferman proved that it is the dual of the real-variable Hardy space $H^1$ (cf. \cite{FeffermanStein}).
 A famous unbounded example of BMO(${\mathbb R}^n$) is $\log |x|$. We refer to Stein \cite{SteinHarmonicBook} for further examples and properties.

For a domain $\Omega\subset {\mathbb C}^n$ we define $PSH^-(\Omega)$ to be the set of negative psh functions on $\Omega$. The purpose of this section is to show the following BMO estimate for psh functions.

 \begin{theorem}\label{th:Main}
 Let $\alpha>1$. If $\psi\in PSH^-(B_{2R})$, then
 \begin{equation}\label{eq:Main}
 \|\psi\|_{{\rm BMO}(B_R)}\le C_{n,\alpha} \left( 1+|\psi_{B_R}|\right)^\alpha
 \end{equation}
  where $C_{n,\alpha}>0$ depends only on $n,\alpha$.
 \end{theorem}

 Theorem \ref{th:Main} will be deduced from a number of lemmas.

  \begin{lemma}\label{lm:Integ_Ineq_mfd}
   Let $\psi,\phi$ be two real $C^2$ function on a domain $\Omega\subset {\mathbb R}^n$. Let $\eta:{\mathbb R}\rightarrow (0,\infty)$ be a $C^1$ function with $\eta'>0$. If either $\phi$ or $\psi$ has compact support in $\Omega$, then
    \begin{equation}\label{eq:Laplace-}
   \int_\Omega {\phi^2}\left[\frac{2\Delta\psi}{\eta (-\psi)} + \frac{\eta'(-\psi)}{\eta^2(-\psi)}{|\nabla\psi|^2}\right]
   \le 4 \int_\Omega \frac{|\nabla\phi|^2}{\eta'(-\psi)}.
  \end{equation}
 \end{lemma}

 \begin{proof}
   Integration by parts gives
  \begin{eqnarray*}
 && \int_\Omega \frac{\phi^2}{\eta (-\psi)}\Delta\psi
 =-\int_\Omega \nabla \psi \cdot \nabla \left[\frac{\phi^2}{\eta(-\psi)}\right]\\
  & = & -2 \int_\Omega \phi \frac{\nabla \psi}{\eta(-\psi)}\cdot \nabla \phi
  - \int_\Omega \phi^2 \frac{\eta'(-\psi)}{\eta^2(-\psi)}|\nabla \psi|^2,
\end{eqnarray*}
so that
\begin{eqnarray*}\label{eq:IntegByParts_1}
 && \int_\Omega \frac{\phi^2}{\eta (-\psi)}\Delta\psi
 +\int_\Omega \phi^2 \frac{\eta'(-\psi)}{\eta^2(-\psi)} |\nabla \psi|^2\\
  & =  &  -2 \int_\Omega \phi \frac{\nabla \psi}{\eta(-\psi)}\cdot \nabla \phi \\
 & \le & \frac12  \int_\Omega  \phi^2 \frac{\eta'(-\psi)}{\eta^2(-\psi)}|\nabla \psi|^2
 + 2 \int_\Omega \frac{|\nabla\phi|^2}{\eta'(-\psi)},
\end{eqnarray*}
from which (\ref{eq:Laplace-}) immediately follows.
\end{proof}

 \begin{lemma}\label{lm:Laplace}
 Let $\alpha>1$. If $\psi$ is a negative subharmonic function on a domain $\Omega\subset {\mathbb R}^n$, then
    $$
 \int_\Omega \phi^2 \Delta \psi\le \frac4{\alpha-1} \int_\Omega (1+|\psi|)^\alpha |\nabla \phi|^2,\ \ \ \phi\in C^\infty_0(\Omega).
 $$
\end{lemma}

\begin{proof}
We take a decreasing sequence of smooth subharmonic functions $\psi_j<0$ defined in a neighborhood of ${\rm supp\,}\phi$ such that $\psi_j\downarrow \psi$.  Applying (\ref{eq:Laplace-}) with $\eta(t)=2-(1+t)^{1-\alpha}$, we have
  $$
  \int_\Omega \frac{\phi^2}{\eta(-\psi_j)}\Delta\psi_j \le \frac2{\alpha-1} \int_\Omega (1+|\psi_j|)^\alpha |\nabla \phi|^2\le \frac2{\alpha-1} \int_\Omega (1+|\psi|)^\alpha |\nabla \phi|^2.
  $$
  As $\eta< 2$, we have
  \begin{eqnarray*}
   \int_\Omega \phi^2 \Delta\psi & = & \lim_{j\rightarrow \infty}  \int_\Omega \phi^2 \Delta\psi_j\\
   & \le & \frac4{\alpha-1} \int_\Omega (1+|\psi|)^\alpha |\nabla \phi|^2.
  \end{eqnarray*}
\end{proof}

\begin{lemma}\label{lm:Doubling}
 If $\psi\in SH^-(2B)$, then
 \begin{equation}\label{eq:Doubling}
 \int_{2B} |\psi| \le 2^n \int_{B} |\psi|.
 \end{equation}
 \end{lemma}

\begin{proof}
Let $\sigma_n$ be the volume of the unit sphere in ${\mathbb R}^n$. Write $B=B(a,r)$. Since $\psi$ is a subharmonic function, it follows that the mean value
$$
M_\psi(a,t)=\int_{|y|=1} \psi(a+ty)d\sigma(y)/\sigma_n
$$
is an increasing function of $t\in (0,2r)$  (see \cite{HormanderConvexity}, Theorem 3.2.3), i.e.
$$
M_{|\psi|} (a,t)=\int_{|y|=1} |\psi| (a+ty)d\sigma(y)/\sigma_n
$$
is a decreasing function of $t$. Then we have
\begin{eqnarray*}
  \int_{2B\backslash B} |\psi|& = & \int_r^{2r} M_{|\psi|} (a,t) t^{n-1}\sigma_n dt \\
  & \le & (2^n-1) \frac{ \sigma_n}{n} r^{n} M_{|\psi|} (a,r)\\
  & \le & (2^n-1) \int_{0}^{r} M_{|\psi|} (a,t) t^{n-1}\sigma_n dt\\
  & = & (2^n-1) \int_{B} |\psi|,
\end{eqnarray*}
from which (\ref{eq:Doubling}) immediately follows.
\end{proof}

  For the proof of Theorem \ref{th:Main} we consider at first the one-dimensional case.

  \begin{lemma}\label{lm:BMO_PSH_1}
  Let $\alpha>1$.  If $n=1$ and $\psi\in SH^-({B_{2R}})$, then
  \begin{equation}\label{eq:BMO_PSH_1}
  \fint_B |\psi-\psi_B|\le C_\alpha \fint_{B_{3R/2}}(1+|\psi|)^\alpha
  \end{equation}
  for all balls $B\subset B_{R}$. Here $C_\alpha>0$ depends only on $\alpha$.
  \end{lemma}

  \begin{proof}
 Applying Lemma \ref{lm:Laplace}  with  $\phi\in C^\infty_0({B_{3R/2}})$ such that $\phi|_{B_{4R/3}}=1$ and $|\nabla \phi|\le 6/R$, we  conclude that
    \begin{equation}\label{eq:BMO_PSH_2}
   \int_{B_{4R/3}} \Delta\psi\le {C_\alpha}  \fint_{B_{3R/2}}(1+|\psi|)^\alpha.
  \end{equation}
  Let $R'=4R/3$. Recall that the (negative) Green function $g_{R'}$ of $B_{R'}$ is given by
  $$
  g_{R'}(z,w)=\log |z-w|+\log \frac{R'}{|R'^2-z\bar{w}|}.
  $$
  The Riesz decomposition theorem (cf. \cite{HormanderConvexity}, Theorem 3.3.6) gives
  \begin{eqnarray*}
   \psi(z) & = & \frac1{2\pi}\int_{\zeta\in B_{R'}} g_{R'}(z,\zeta)\Delta \psi(\zeta) +h(z)\\
   &  = & \frac1{2\pi} \int_{\zeta\in B_{R'}} \log |z-\zeta| \Delta \psi(\zeta)+\frac1{2\pi}\int_{\zeta\in B_{R'}} \log \frac{R'}{|R'^2-z\bar{\zeta}|}\Delta \psi(\zeta) + h(z)\\
   & = :& u(z)+v(z)+h(z)
  \end{eqnarray*}
  where $h$ is the\/ {\it smallest}\/ harmonic majorant of $\psi$, which naturally satisfies
  $$
  \psi\le h\le 0.
  $$
    Since $h\le 0$ is harmonic on $B_{R'}$, it follows from the mean value property that for each $z\in B_{R}$
  \begin{eqnarray}\label{eq:BMO_PSH_4}
  -h(z) & = &¡¡\fint_{B(z,R/4)} (-h)\nonumber\\
   & \le & \frac1{\pi (R/4)^{2}} \int_{B_{R'}} (-\psi)\nonumber\\
   & \le & C_\alpha  \fint_{B_{3R/2}}(1+|\psi|)^\alpha.
  \end{eqnarray}
  For each ball $B\subset B_{R}$ and $z\in B$, we have
  \begin{eqnarray*}
  2\pi [u(z)-u_B] & = & \int_{\zeta\in B_{R'}}\log |z-\zeta|\, \Delta\psi(\zeta)-\frac1{|B|}\int_{w\in B} \int_{\zeta\in B_{R'}} \log |w-\zeta|\,\Delta\psi(\zeta)\\
   & = & \int_{\zeta\in B_{R'}}[\log |z-\zeta|-(\log |\cdot-\zeta|)_B] \Delta\psi(\zeta).
  \end{eqnarray*}
  As the BMO norm on ${\mathbb C}^n$ is invariant under translations, it follows from Fubini's theorem that
   \begin{eqnarray}\label{eq:BMO_PSH_5}
  \fint_B |u-u_B| & \le & \frac1{2\pi}\|\log |z|\,\|_{{\rm BMO}({\mathbb C}^n)}  \int_{B_{R'}} \Delta\psi\nonumber\\
   & \le & C_\alpha  \fint_{B_{3R/2}}(1+|\psi|)^\alpha.
  \end{eqnarray}
  Analogously, as
  $$
  \log \frac{R'}{|R'^2-z\bar{\zeta}|}=\log \frac{R'}{|\zeta|}-\log \left|z-R'^2/\bar{\zeta}\right|
  $$
  for $\zeta\neq 0$, it follows that the BMO norms (in $z$) of $ \log \frac{R'}{|R'^2-z\bar{\zeta}|}$ and $-\log |z|$ coincide, while for $\zeta=0$, $ \log \frac{R'}{|R'^2-z\bar{\zeta}|}\equiv \log 1/R'$, so that its BMO norm is zero. Thus
   \begin{equation}\label{eq:BMO_PSH_3}
   \fint_B |v-v_B|
  \le C_\alpha  \fint_{B_{3R/2}}(1+|\psi|)^\alpha.
  \end{equation}

  Clearly, (\ref{eq:BMO_PSH_4})-(\ref{eq:BMO_PSH_3}) imply (\ref{eq:BMO_PSH_1}).
  \end{proof}

  For each $a\in {\mathbb C}^n$ and each ${\rm r}=(r_1,\cdots,r_n)$ where $r_j>0$, we define the polydisc
  $$
  P(a,{\rm r})=\left\{z\in {\mathbb C}^n:  |z_j-a_j|<r_j,\,1\le j\le n \right\}.
  $$
  Set $P(a,r)=P(a,(r,\cdots,r))$. Then we have

  \begin{lemma}\label{lm:BMO_PSH_2}
   Let $\alpha\ge 1$.  If $\psi\in PSH^-({P(0,2{\rm r})})$, then
         $$
   \fint_{P(0,{\rm r})}|\psi|^\alpha \le C_{n,\alpha}  \fint_{|z_k|<r_k} |\psi(0,\cdots,0,z_k,0,\cdots,0)|^\alpha
   $$
  for all $1\le k\le n$, where $C_{n,\alpha}>0$ depends only on $n,\alpha$.
  \end{lemma}

  \begin{proof}
 It suffices to consider the case $k=1$. The Riesz decomposition theorem implies that if $u<0$  is a subharmonic function in a neighborhood of the unit closed disc in ${\mathbb C}$ then
  \begin{equation}\label{eq:Hormander}
  \int_{|z|<3/4} |u|^\alpha \le C_\alpha |u(0)|^\alpha
  \end{equation}
  where $C_\alpha$ depends only on $\alpha$  (see \cite{HormanderConvexity}, p.\,230). In the case of $n$ complex variables we consider a negative psh function $u$ in a neighborhood of the unit closed polydisc in ${\mathbb C}^n$. Then we have
  \begin{eqnarray*}
   \int_{P(0,3/4)} |u|^\alpha & \le & C_\alpha \int_{|z_1|<3/4}\cdots\int_{|z_{n-1}|<3/4} |u(z_1,\cdots,z_{n-1},0)|^\alpha\\
   & \le & \cdots \le C_\alpha^{n-1} \int_{|z_1|<3/4} |u(z_1,0')|^\alpha,
  \end{eqnarray*}
  so that
  $$
   \fint_{P(0,3/4)} |u|^\alpha \le C_{n,\alpha} \fint_{|z_1|<3/4} |u(z_1,0')|^\alpha.
  $$
  It suffices to apply the above inequality with $u(z)=\psi(4r_1 z_1/3,\cdots, 4 r_n z_n/3)$.
  \end{proof}

  \begin{lemma}\label{lm:BMO_PSH_3}
  Let $\alpha>1$. If $\psi\in PSH^-({P(0,2{\rm  R})})$, then
  \begin{equation}\label{eq:BMO_PSH_6}
  \fint_P |\psi-\psi_P|\le C_{n,\alpha}   \sum_{k=1}^n \fint_{|z_k|< 3R_k/2}\left(1+|\psi(0,\cdots,0,z_k,0,\cdots,0)|\right)^\alpha
  \end{equation}
  for any polydisc $P=P(0,{\rm r})$  with ${\rm r}\le {\rm R}$, i.e. $r_k\le R_k$ for all $k$.
  \end{lemma}

  \begin{proof}
  We write
  $
  P=\prod_{j=1}^n B^j
  $
  where $B^j=\{z_j:|z_j|<r_j\}$. For each $z\in P$ we have
  \begin{eqnarray*}
  \psi(z)-\psi_P & = & \psi(z_1,z_2,\cdots,z_n)-\psi(\cdot,z_2,\cdots,z_n)_{B^1}+\cdots\\
  && + \psi(\cdots,z_{k},\cdots,z_n)_{B^1\cdots B^{k-1}}- \psi(\cdots,z_{k+1},\cdots,z_n)_{B^1\cdots B^{k}}+\cdots\\
  && +  \psi(\cdots,z_n)_{B^1\cdots B^{n-1}}- \psi_{B^1\cdots B^{n}}
  \end{eqnarray*}
  where
 $$
   \psi(\cdots,z_{k},\cdots,z_n)_{B^1\cdots B^{k-1}}
   = \fint_{\zeta_1\in B^1}\cdots \fint_{\zeta_{k-1}\in B^{k-1}}\psi(\zeta_1,\cdots,\zeta_{k-1},z_k\cdots,z_n).
 $$
  Since
  \begin{eqnarray*}
  && \left|\psi(\cdots,z_{k},\cdots,z_n)_{B^1\cdots B^{k-1}}- \psi(\cdots,z_{k+1},\cdots,z_n)_{B^1\cdots B^{k}}\right|\\
  & \le & \left|\psi(\cdots,z_{k},\cdots,z_n)- \psi(\cdots,z_{k+1},\cdots,z_n)_{B^{k}}\right|_{B^1\cdots B^{k-1}},
  \end{eqnarray*}
  it follows that
   \begin{eqnarray*}
  &&  \fint_{z_k\in B^k} \left|\psi(\cdots,z_{k},\cdots,z_n)_{B^1\cdots B^{k-1}}- \psi(\cdots,z_{k+1},\cdots,z_n)_{B^1\cdots B^{k}}\right|\\
  & \le & \left|  \fint_{z_k\in B^k} \left|\psi(\cdots,z_{k},\cdots,z_n)- \psi(\cdots,z_{k+1},\cdots,z_n)_{B^{k}}\right| \right|_{B^1\cdots B^{k-1}}\\
  & \le &  C_\alpha   \left[\fint_{|z_k|<3R_k/2} \left(1+|\psi(\cdots,z_{k},\cdots,z_n)|\right)^\alpha \right]_{B^1\cdots B^{k-1}}
    \end{eqnarray*}
    in view of Lemma \ref{lm:BMO_PSH_1}.
    Thus
    \begin{eqnarray*}
     \fint_P \left|\psi-\psi_P\right|
     & \le &  C_\alpha \sum_{k=1}^n
     \left[\fint_{|z_k|<3R_k/2} \left(1-\psi(\cdots,z_{k},\cdots)\right)^\alpha \right]_{B^1\cdots B^{k-1}B^{k+1}\cdots B^n}\\
     & \le &  C_{n,\alpha}    \sum_{k=1}^n \fint_{|z_k|<3R_k/2}\left(1-\psi(0,\cdots,0,z_k,0,\cdots,0)\right)^\alpha
    \end{eqnarray*}
  in view of Lemma \ref{lm:BMO_PSH_2}, for $\psi-1$ is a negative psh function.
  \end{proof}

 \begin{lemma}\label{lm:BMO_PSH_5}
 If  $\psi\in PSH^-(B_R)$, then there exists  $C_n>0$ depending only on $n$,  such that there are complex lines $L_1,\cdots,L_n$ through the origin,  which are orthogonal each other and satisfy
  $
  L_j\cap S\neq \emptyset
  $
  for all $j$, where
   $$
  S=\left\{z\in B_{R/3^n}:\psi(z)>C_n\psi_{B_{R/3}}\right\}.
  $$
  \end{lemma}

  \begin{proof}
  Set
  $$
  S_m=\left\{z\in B_{R/3}:\psi(z)>m\psi_{B_{R/3}}\right\},\ \ \ S_m^c=B_{R/3}-S_m.
    $$
   By Chebychev's inequality
  $$
  \int_{B_{R/3}} |\psi|\ge -m \psi_{B_{R/3}} |S_m^c|,
  $$
  we establish
  \begin{eqnarray}\label{eq:VolumeEstimate}
  |S_m|  \ge  |B_{R/3}|- |{B_{R/3}} |/m
   >  |B_{R/3}|/2
  \end{eqnarray}
  provided
  $
  m>2.
  $
  We  choose a universal constant $0<c_1< 1/3$ such that
  $$
  \left|\{z\in B_{R/3}:|z_1|<c_1 R\}\right|+\left|\{z\in B_{R/3}:|z_2|<c_1 R\}\right|<|B_{R/3}|/4.
  $$
  Set
  $$
  S_m'=S_m\cap \{z:\min\{|z_1|,|z_2|\}>c_1 R\}.
  $$
  Then we have
  $$
  |S_m'|\ge |S_m|-|B_{R/3}|/4> |S_m|/2
  $$
  in view of (\ref{eq:VolumeEstimate}).
  We define a diffeomorphism $F$ on $S_m'$ as follows: $w_j=z_j$ for $j>1$ and
  $$
  w_1=-(|z_2|^2+\cdots+|z_n|^2)/\bar{z}_1.
  $$
  Clearly, the vector $F(z)$ is orthogonal to $z$ in ${\mathbb C}^n$ and satisfies
  $$
  |F(z)|\le {|z|^2}/{|z_1|}< (9c_1)^{-1}R,
  $$
  i.e. $(3c_1)\cdot F(z)\in B_{R/3}$.
    Since the real Jacobian $J_{\mathbb R}(F)$ of $F$ satisfies
  $$
  J_{\mathbb R}(F)(z)=-(|z_2|^2+\cdots+|z_n|^2)^2/|z_1|^4
  $$
  for $z\in S_m'$, it follows that $|J_{\mathbb R}(F)(z)|\ge (3c_1)^4$ and
  $$
  |(3c_1)\cdot F(S_m')|\ge ( 3c_1)^{2n+4} |S_m'|> \frac12 (3c_1)^{2n+4} |S_m|.
  $$
  Thus if we choose
  $$
   m>\frac{1+\frac12 (3c_1)^{2n+4}}{\frac12 (3c_1)^{2n+4}}
     $$
  so that
  $$
  |S_m|>\left[1+\frac12 (3c_1)^{2n+4}\right]^{-1} |B_{R/3}|
  $$
  in view of (\ref{eq:VolumeEstimate}),
   then
   $$
   S_m\cap \left[(3c_1)\cdot F(S_m')\right]\neq \emptyset.
   $$
    In other words, there exists a complex line $L_1$ such that both $L_1$ and its orthogonal complement $L_1^\bot$ in ${\mathbb C}^n$ intersect $S_m$. Suppose $a\in S_m\cap L^\bot_1$. The mean value inequality for the psh function $\psi$ implies
  $$
  \fint_{B(a,2R/3)\cap L^\bot_1} |\psi|\le |\psi(a)|.
  $$
  Since $B_{R/3}\subset B(a,2R/3)$, we have
   $$
  \fint_{B_{R/3}\cap L^\bot_1} |\psi|\le C_n m \fint_{B_{R/3}} |\psi|
  $$
  where $C_n>0$ depends only on $n$.
  By repeating the previous argument, we obtain the remaining complex lines $L_2,\cdots,L_n$.
  \end{proof}

  \begin{proof}[Proof of Theorem \ref{th:Main}]
 Given $a\in B_R$, we have $B(a,R)\subset B_{2R}$ and
 $$
 B(a,r)\subset P(a,r)\subset B(a,{R/3}),\ \ \ r\le R/(3{n}^{1/2}).
 $$
 We assume $a=0$ for the sake of simplicity. Let $L_j$, $1\le j\le n$, be chosen as Lemma \ref{lm:BMO_PSH_5}. By a unitary transformation, we may assume that $L_j$ is the $z_j-$axis for each $j$. By Lemma \ref{lm:BMO_PSH_3} and Lemma \ref{lm:BMO_PSH_4}, we see that
  $$
  \fint_{B(0,r)} \left|\psi-\psi_{B(0,r)}\right|\le C_{n,\alpha}    \sum_{k=1}^n \fint_{|z_k|<R/2}\left(1+|\psi(0,\cdots,0,z_k,0,\cdots,0)|\right)^\alpha.
  $$
  Let $b^{(k)}=(0,\cdots,0,b_k,0,\cdots,0)\in L_k\cap S$. It follows from (\ref{eq:Hormander}) and Lemma \ref{lm:BMO_PSH_5} that
  $$
  \fint_{|z_k-b_k|<R/2+R/3^n}|\psi(0,\cdots,0,z_k,0,\cdots,0)|^\alpha\le C_{n,\alpha}  |\psi(b^{(k)})|^\alpha\le C_{n,\alpha} |\psi_{B(0,{R/3})}|^\alpha.
  $$
 As
 $$
\{z_k:|z_k|<R/2\}\subset \{z_k:|z_k-b_k|<R/2+R/3^n\},
 $$
 we have
    $$
   \fint_{|z_k|<R/2}\left|\psi(0,\cdots,0,z_k,0,\cdots,0)\right|^\alpha\le C_{n,\alpha} |\psi_{B(0,{R/3})}|^\alpha\le C_{n,\alpha} |\psi_{B_{2R}}|^\alpha\le C_{n,\alpha} |\psi_{B_{R}}|^\alpha
     $$
    in view of Lemma \ref{lm:BMO_PSH_4}  and Lemma \ref{lm:Doubling}. Thus
     $$
      \fint_{B(0,r)} \left|\psi-\psi_{B(0,r)}\right|\le C_{n,\alpha} \left(1+|\psi_{B_{R}}|\right)^\alpha.
                $$
  On the other hand, for each ball $B(a,r)\subset B_R$ with $r>R/(3{n}^{1/2})$, we naturally have
  \begin{eqnarray*}
   \fint_{B(a,r)} \left|\psi-\psi_{B(a,r)}\right| & \le & 2|\psi_{B(a,r)}|\le 2 \left(1+|\psi_{B(a,r)}|\right)^\alpha\\
   & \le & C_{n,\alpha} \left(1+|\psi_{B_{2R}}|\right)^\alpha \\
   & \le & C_{n,\alpha} \left(1+|\psi_{B_{R}}|\right)^\alpha
        \end{eqnarray*}
        in view of Lemma \ref{lm:BMO_PSH_4}.
        Thus we have (\ref{eq:Main}).
    \end{proof}

    \section{Proof of Theorem \ref{th:JohnNirenberg}}

    Let us recall some basic facts from the theory of weights, by following Stein \cite{SteinHarmonicBook}. A local integrable function $\omega\ge 0$ on a domain $\Omega$ in ${\mathbb R}^n$ is said to satisfy the $A_p$ condition if
    \begin{equation}\label{eq:A_pCondition}
     \left[\fint_B \omega \right]\cdot \left[\fint_B \omega^{-1/(p-1)}\right]^{p-1}\le A<\infty
    \end{equation}
    for all balls $B\subset\subset \Omega$.  The smallest constant $A$ for which (\ref{eq:A_pCondition}) holds is called the $A_p$ constant of $\omega$, which is denoted by $A_p(\omega)$. It is known that $\omega\in A_p$ if and only if
    \begin{equation}\label{eq:A_pEquivalent}
     (f_B)^p\le C \left[\int_B f^p \omega\right] \cdot \left[\int_B \omega\right]^{-1}
    \end{equation}
    for all nonnegative $f\in L^1_{\rm loc}(\Omega)$ and all balls $B\subset\subset \Omega$; moreover the smallest $C$ for which (\ref{eq:A_pEquivalent}) is valid equals $A_p(\omega)$ (see \cite{SteinHarmonicBook}, p.\,195). Let $E$ be a measurable set in $B$ and $\chi_E$ the characteristic function of $E$. Applying (\ref{eq:A_pEquivalent}) with $f=\chi_E$ we establish
    \begin{equation}\label{eq:A_pVariation}
     \int_B \omega \le A_p(\omega)\left(|B|/|E|\right)^p \int_E \omega.
    \end{equation}
    In particular, $\omega$ satisfies a doubling property
    \begin{equation}\label{eq:A_pDoubling}
    \int_{B}\omega \le 2^{np}A_p(\omega)\int_{\frac12 B} \omega.
    \end{equation}

       Let $\psi\in {\rm BMO}(\Omega)$ and $u:=c_n \psi/\|\psi\|_{{\rm BMO}(\Omega)}$, where $c_n$ is the constant in  (\ref{eq:JohnNirenberger}). Then we have
    $$
    \fint_B e^{u-u_B}\le C_n,\ \ \ \fint_B e^{u_B-u}\le C_n,
    $$
    so that
    \begin{equation}\label{eq:A_2}
    \left[\fint_B e^u\right]\cdot  \left[\fint_B e^{-u}\right]=\left[\fint_B e^{u-u_B}\right]\cdot \left[\fint_B e^{u_B-u}\right]\le C_n^2,
    \end{equation}
    i.e. $e^u,\,e^{-u}\in A_2$. Applying (\ref{eq:A_pVariation}) with $\omega=e^u$  we establish
     \begin{equation}\label{eq:A_2Variation}
     \int_B e^u \le C_n \left(|B|/|E|\right)^2 \int_E e^u.
    \end{equation}

     Now we can prove the following inequality mentioned in \S\,1.

     \begin{theorem}\label{th:Bernstein}
   Let $B_R=\{z\in {\mathbb C}^n:|z|<R\}$ and $\alpha>1$. If $\psi\in PSH^-(B_{2R})$, then for each ball $B\subset B_R$ and each measurable set $E\subset B$ one has
   \begin{equation}\label{eq:Bernstein}
   \sup_B \psi\le \sup_E \psi+ C_{n,\alpha} \left( 1+|\psi_{B_R}|\right)^\alpha \left[1+\log \left({|B|}/{|E|}\right)\right]
   \end{equation}
   where $C_{n,\alpha}>0$ depends only on $n,\alpha$.
   \end{theorem}

        \begin{proof}
        Let $B$ be a ball in $B_R$ and $E$ a measurable set in $B$.  Then we have
        \begin{equation}\label{eq:Bernstein_2}
        \int_E e^{u}\le |E|\, e^{\sup_E u}
        \end{equation}
        where $u=c_n\psi/\|\psi\|_{{\rm BMO}(B_R)}$.
        On the other hand, we choose a point $a\in \overline{B}$ such that $u(a)=\sup_B u$. Let $r$ be the radius of $B$. The doubling property (\ref{eq:A_pDoubling}) implies
        \begin{equation}\label{eq:Bernstein_3}
        \int_B e^{u}\ge C_n^{-1}\int_{2B} e^{u}\ge C_n^{-1}\int_{B(a,r)} e^{u}\ge C_n^{-1} |B(a,r)|e^{u(a)}\ge C_n^{-1}|B|\,e^{\sup_B u}
        \end{equation}
        where the third inequality follows from the mean value inequality for the psh function $e^u$. Combining  (\ref{eq:A_2Variation}), (\ref{eq:Bernstein_2}) and (\ref{eq:Bernstein_3}) yields
        \begin{equation}\label{eq:Bernstein_4}
        \sup_B u\le \sup_E u + \log \left({|B|}/{|E|}\right)+C_n.
        \end{equation}
       This inequality combined with Theorem \ref{th:Main} gives (\ref{eq:Bernstein}).
        \end{proof}

         Theorem \ref{th:Bernstein} implies a new interpretation of the Lelong number.

   \begin{corollary}\label{cor:Lelong}
   Let $\psi$ be a psh function on a domain $\Omega\subset {\mathbb C}^n$. Let $\nu(\psi,z)$ denote the Lelong number of $\psi$ at $z\in \Omega$. Let $E_r,\,0<r\ll1$, be a family of measurable sets satisfying $E_r\subset B(z,r)$ and
   $$
   \log \left(|B(z,r)|/|E_r|\right)={\rm o}(\log 1/r),\ \ \ r\rightarrow 0.
   $$
  Then we have
   \begin{equation}\label{eq:Lelong}
   \nu(\psi,z) =\lim_{r\rightarrow 0}{({\sup}_{E_r}\psi)}/{\log r}.
   \end{equation}
   \end{corollary}

          \begin{proof}
    After subtracting a constant to $\psi$ we may assume $\psi<0$ on $B(z,r_0)$ for some $r_0<d(z,\partial \Omega)$. As
    $$
    \nu(\psi,z)=\lim_{r\rightarrow 0}\, ({\sup}_{B(z,r)} \psi)/\log r,
   $$
     we have
     $$
    \nu(\psi,z)\le \liminf_{r\rightarrow 0}\, ({\sup}_{E_r} \psi)/\log r.
    $$
    On the other hand, (\ref{eq:Bernstein}) implies
       $$
   \nu(\psi,z)\ge \limsup_{r\rightarrow 0} \, ({\sup}_{E_r} \psi)/\log r.
   $$
       \end{proof}

        \begin{theorem}\label{th:JNExponent_2}
 Let $\alpha>2$ and $\gamma>1$. If $\psi\in PSH^-(\tilde{B}_{2R})$, then there exist positive constants $c_{n,\alpha,\gamma}$ depends only on $n,\alpha,\gamma$ and $C_n$ depending only on $n$ such that
 \begin{equation}\label{eq:JNexponent_2}
\fint_{\tilde{B}_r+a} e^{\varepsilon|\psi-\psi_{\tilde{B}_r+a}|} \le C_n
 \end{equation}
 for all $\tilde{B}_r+a:=\{z+a:z\in \tilde{B}_r\}\subset \tilde{B}_R$,
 where
 $$
 \varepsilon=c_{n,\alpha,\gamma} \left[1+ \fint_{\tilde{B}_{3R/2}} | \psi|^\alpha  \right]^{-\gamma/\alpha}.
 $$
  \end{theorem}

  Let us first observe that Theorem \ref{th:JohnNirenberg} follows from Theorem \ref{th:JNExponent_2}. Let $U$ be a neighborhood of $\overline{\Omega}$ such that $\psi$ is psh on $\overline{U}$.  Let $D\in \tilde{\mathcal B}(\Omega)$. By a change of complex coordinates we may assume that $D$ is of form $\tilde{B}_r$ for some $r>0$. As $\tilde{B}_{2R}\subset U$ for $R:=d(D,\partial U)^2/4$ (assume $d(D,\partial U)\le 1$ for the sake of simplicity), we apply Theorem \ref{th:JNExponent_2} with $\psi$ replaced by $\tilde{\psi}:=\psi-\sup_U \psi$ to get
  $$
  \fint_{D} e^{\varepsilon |\psi-\psi_{D}|} \le C_n
  $$
  provided
  \begin{eqnarray*}
    \varepsilon & = & c_{n,\alpha,\gamma} \left[1+ \fint_{\tilde{B}_{3R/2}} |\tilde{ \psi}|^\alpha  \right]^{-\gamma/\alpha}\\
    & \ge & c_{n,\alpha,\gamma}\left[1+ d(D,\partial U)^{-2(n+1)} \int_{U} |\tilde{ \psi}|^\alpha  \right]^{-\gamma/\alpha}.
   \end{eqnarray*}
   This completes the proof of Theorem \ref{th:JohnNirenberg}.

   Theorem \ref{th:JNExponent_2} will be deduced from the following inequalities.

   \begin{lemma}\label{lm:ReverseHolder}
   Let $\alpha>2$ and $\gamma>1$. If $\psi\in PSH^-(\tilde{B}_{2R})$, then there exists a number
   $$
   0<\lambda\le C_{n,\alpha,\gamma}\left[1+\fint_{\tilde{B}_{3R/2}} | \psi|^\alpha\right]^{\gamma/\alpha}
   $$
   such that  for each $r\le R/2$ and $a$ with $\tilde{B}_r+a\subset \tilde{B}_R$,
   \begin{equation}\label{eq:ReverseHolder}
    \left[\fint_{\tilde{B}_r+a}  e^{2\psi/\lambda}\right]^{1/2}\le C_n \fint_{\tilde{B}_{r}+a}  e^{\psi/\lambda}
   \end{equation}
   \begin{equation}\label{eq:Doubling_2}
    \int_{\tilde{B}_r+a}  e^{\psi/\lambda}\le C_n \int_{\tilde{B}_{r/2}+a}  e^{\psi/\lambda}.
   \end{equation}
      \end{lemma}

      \begin{proof}
   For each $r$ and $a$ with $\tilde{B}_r+a\subset \tilde{B}_R$ we define $R_{a,r}$ to be the supremum of all $t\ge r$ such that
     $$
  \tilde{B}_{t}+a\subset \tilde{B}_{3R/2}.
  $$
  Clearly we have $c_n R\le R_{a,r}\le 3R/2$ and
  $$
  \fint_{\tilde{B}_{R_{a,r}}+a} | \psi|^\alpha \le C_n  \fint_{\tilde{B}_{3R/2}} | \psi|^\alpha.
  $$
  It suffices to verify (\ref{eq:ReverseHolder}) and (\ref{eq:Doubling_2})  with
    $$
    \lambda \le C_{n,\alpha,\gamma}\left[1+\fint_{{\tilde{B}_{R_{a,r}}+a}} | \psi|^\alpha\right]^{\gamma/\alpha}.
    $$
  For the sake of simplicity we assume $a=0$ and write $R_{a,r}$ as $R$.
   Set $\varphi(\zeta)=\psi(\zeta_1^2,\zeta')$  and
   $$
   B^\ast_r=\left\{\zeta\in {\mathbb C}^n: |\zeta_1|<r,|\zeta'|<r\right\}.
   $$
    Let $r'=\sqrt{r}$ and $\gamma>1$.   By Theorem \ref{th:Bernstein} we conclude that if $r'\le R'/\sqrt{2}$ then
   \begin{eqnarray*}
   \sup_{B^\ast_{r'}} \varphi  \le  \sup_{B(0,\sqrt{2}r')} \varphi & \le & \sup_{B^\ast_{r'/2}} \varphi +C_{n,\gamma} \left(1+|\varphi|_{B(0,R')}\right)^\gamma \left[1+\log \frac{|B(0,\sqrt{2}r')|}{|B^\ast_{r'/2}|}\right]\\
   & \le &  \sup_{B^\ast_{r'/2}} \varphi +C_{n,\gamma} \left(1+|\varphi|_{B^\ast_{R'}}\right)^\gamma\\
    & =: & \sup_{B^\ast_{r'/2}} \varphi+\lambda.
   \end{eqnarray*}
   It follows  that
   $$
    \sup_{B^\ast_{r'}} e^{\varphi/\lambda} \le e \sup_{B^\ast_{r'/2}} e^{\varphi/\lambda},
   $$
  i.e.
   $$
    \sup_{\tilde{B}_r} e^{\psi/\lambda} \le e  \sup_{\tilde{B}_{r/4}} e^{\psi/\lambda}.
   $$
   Then we have
   $$
   \left[\fint_{\tilde{B}_r} e^{2\psi/\lambda}\right]^{1/2}\le  \sup_{\tilde{B}_r} e^{\psi/\lambda}\le e \sup_{\tilde{B}_{r/4}} e^{\psi/\lambda}\le C_n   \fint_{\tilde{B}_r} e^{\psi/\lambda}
   $$
   $$
   \fint_{\tilde{B}_r} e^{2\psi/\lambda} \le  \sup_{\tilde{B}_r} e^{\psi/\lambda}\le e \sup_{\tilde{B}_{r/4}} e^{\psi/\lambda}\le C_n   \fint_{\tilde{B}_{r/2}} e^{\psi/\lambda}
   $$
   in view of the mean value inequality for the psh function $e^{\psi/\lambda}$.

   Let $\alpha'$ be the dual exponent of $\alpha$.
  As
  \begin{eqnarray*}
   \fint_{B^\ast_{R'}} | \varphi| & \le & \frac1{|B^\ast_{R'}|} \left[\int_{B^\ast_{R'}}  |\zeta_1|^{-2\alpha'/\alpha}\right]^{1/\alpha'} \left[\int_{B^\ast_{R'}} | \varphi|^{\alpha} |\zeta_1|^2\right]^{1/\alpha}\\
   & = & \frac1{|B^\ast_{R'}|}\left[\int_{B^\ast_{R'}}  |\zeta_1|^{-2\alpha'/\alpha}\right]^{1/\alpha'} \left[\frac12\int_{\tilde{B}_R} | \psi|^\alpha \right]^{1/\alpha}\\
   & \le & C_{n,\alpha} R^{-(n+1)/\alpha} \left[\int_{\tilde{B}_R} | \psi|^\alpha \right]^{1/\alpha}\\
   & \le & C_{n,\alpha} \left[\fint_{\tilde{B}_R} | \psi|^\alpha \right]^{1/\alpha},
  \end{eqnarray*}
  we have
  $$
  \lambda\le C_{n,\alpha,\gamma}\left[1+\fint_{\tilde{B}_R} | \psi|^\alpha\right]^{\gamma/\alpha}.
  $$
   \end{proof}

   \begin{proof}[Proof of Theorem \ref{th:JNExponent_2}]
   It is a standard fact that the reverse H\"older inequality like (\ref{eq:ReverseHolder}) and the doubling property like (\ref{eq:Doubling_2}) for Euclidean balls would imply the $A_p$ property for some $p>1$. We shall show that the same is true for\/ {\it nonisotropic} \/balls $\tilde{B}_r+a$ by using Calder\'on's work \cite{Calderon}. We define
   $$
   \varrho(z,w)=\max\left\{|z_1-w_1|,|z'-w'|^2\right\},\ \ \ z,w\in {\mathbb C}^n.
   $$
   It is easy to verify that $\varrho$ satisfies the following properties
   \begin{enumerate}
   \item $\varrho(z,z)=0$;
   \item $\varrho(z,w)=\varrho(w,z)>0$ if $z\neq w$;
   \item $\varrho(z,w)\le 2(\varrho(z,\zeta)+\varrho(\zeta,w))$ for all $z,w,\zeta \in {\mathbb C}^n$.
   \end{enumerate}
   Note also that $\tilde{B}_r+a=\{z:\varrho(z,a)<r\}$. Set $\omega=e^{\psi/\lambda}$ and $d\mu=\omega dV$ where $dV$ is the Lebesgue measure in ${\mathbb C}^n$. Let $|\cdot|_\mu$ be the volume associated to $d\mu$. Then we may rewrite (\ref{eq:Doubling_2}) as
   \begin{equation}\label{eq:Doubling_3}
    \left|\tilde{B}_r+a\right|_\mu  \le C_n \left|\tilde{B}_{r/2}+a\right|_\mu.
   \end{equation}
   Let $E$ be a measurable set in $\tilde{B}_r+a$. By (\ref{eq:ReverseHolder}) we have
   \begin{eqnarray*}
    |E|_\mu = \int_E \omega & \le & \left[\int_E\omega^2\right]^{1/2}|E|^{1/2}\\
    & \le & \left[\int_{\tilde{B}_r+a}\omega^2\right]^{1/2}|E|^{1/2}\\
    & \le & C_n \left|\tilde{B}_r+a\right|^{-1/2} \left[\int_{\tilde{B}_r+a}\omega\right] |E|^{1/2},
   \end{eqnarray*}
   i.e.
  $$
    \frac{|E|}{\left|\tilde{B}_r+a\right|}\ge C_n^{-1}\left[\frac{|E|_\mu}{\left|\tilde{B}_r+a\right|_\mu}\right]^2.
  $$
   According to Calder\'on (see \cite{Calderon}, the proof of Theorem 1), the above inequality implies a reverse H\"older inequality w.r.t. the measure $d\mu$ (noting that $dV=\omega^{-1}d\mu$)
   $$
   \left[ \frac1{\left|\tilde{B}_r+a\right|_\mu}\int_{\tilde{B}_r+a}\omega^{-p_n} d\mu \right]^{1/p_n} \le C_n \, \frac1{\left|\tilde{B}_r+a\right|_\mu} \int_{\tilde{B}_r+a}\omega^{-1} d\mu,
   $$
   for some $p_n> 1$, which may be rewritten as
   $$
   \left[\fint_{\tilde{B}_r+a}\omega\right]\cdot \left[\fint_{\tilde{B}_r+a}\omega^{-p_n+1}\right]^{1/(p_n-1)}\le C_n.
   $$
   It follows that
   $$
    \left[\fint_{\tilde{B}_r+a} e^{(\psi-\psi_{\tilde{B}_r+a})/\lambda}\right]\cdot \left[\fint_{\tilde{B}_r+a}e^{-(p_n-1)(\psi-\psi_{\tilde{B}_r+a})/\lambda}\right]^{1/(p_n-1)}\le C_n.
   $$
   This inequality combined with Jensen's inequality
   $$
   \fint_{\tilde{B}_r+a} e^{(\psi-\psi_{\tilde{B}_r+a})/\lambda}\ge 1 \ \ \ {\rm and\ \ \ } \fint_{\tilde{B}_r+a}e^{-(p_n-1)(\psi-\psi_{\tilde{B}_r+a})/\lambda}\ge 1
   $$
   yields
   $$
   \fint_{\tilde{B}_r+a} e^{(\psi-\psi_{\tilde{B}_r+a})/\lambda}\le C_n \ \ \ {\rm and\ \ \ } \fint_{\tilde{B}_r+a}e^{-(p_n-1)(\psi-\psi_{\tilde{B}_r+a})/\lambda}\le C_n.
   $$
   Set $u=(\psi-\psi_{\tilde{B}_r+a})/\lambda$ and $\varepsilon_n=\min\{1,p_n-1\}$. Then we have
   \begin{eqnarray*}
     \fint_{\tilde{B}_r+a} e^{\varepsilon_n |\psi-\psi_{\tilde{B}_r+a}|/\lambda} & = & \frac1{|\tilde{B}_r+a|}\int_{\{u\ge 0\}} e^{\varepsilon_n u}+\frac1{|\tilde{B}_r+a|}\int_{\{u < 0\}} e^{-\varepsilon_n u}\\
     & \le & \fint_{\tilde{B}_r+a} e^{ u}+\fint_{\tilde{B}_r+a} e^{-(p_n-1) u}\\
     & \le & C_n.
   \end{eqnarray*}
   \end{proof}

\end{document}